\documentclass[12pt, reqno]{amsart}
\usepackage{eurosym}
\usepackage{amsmath, amsthm, amscd, amsfonts, amssymb, graphicx, color}
\usepackage[bookmarksnumbered, colorlinks, plainpages]{hyperref}

\hypersetup{colorlinks=true,linkcolor=red, anchorcolor=green, citecolor=cyan, urlcolor=red, filecolor=magenta, pdftoolbar=true}
\newtheorem{theorem}{Theorem}[section]
\newtheorem{lemma}[theorem]{Lemma}
\newtheorem{proposition}[theorem]{Proposition}
\newtheorem{corollary}[theorem]{Corollary}
\theoremstyle{definition}

\theoremstyle{remark}

\numberwithin{equation}{section}

\begin{document}
\title[Higher order generalization of Bernstein type operators defined by $%
(p,q)$-integers]{Some approximation results on higher order generalization
of Bernstein type operators defined by $(p,q)$-integers}
\author[M. Mursaleen and Md. Nasiruzzaman]{M. Mursaleen and Md. Nasiruzzaman}
\address{Department of Mathematics, Aligarh Muslim University,
Aligarh--202002, India.}
\email{%
\textcolor[rgb]{0.00,0.00,0.84}{mursaleenm@gmail.com;
nasir3489@gmail.com}}
\footnote{Corresponding author. Email: mursaleenm@gmail.com; Phone: +919411491600}

\begin{abstract}
In this paper, we introduce the higher order generalization of Bernstein
type operators defined by $(p,q)$-integers. We establish some approximation
results for these new operators by using the modulus of continuity.
\end{abstract}

\maketitle

\setcounter{page}{1}


\section{Introduction and preliminaries}

In 1912, S.N Bernstein \cite{n21} introduced the following sequence of
operators $B_{n}:C[0,1]\rightarrow C[0,1]$ defined for any $n\in \mathbb{N}$
and for any $f\in C[0,1]$ such as
\begin{align}  \label{eq1}
B_{n}(f;x)=\sum_{k=0}^{n}\binom{n}{k}x^{k}(1-x)^{n-k}f\left( \frac{k}{n}%
\right) ,~~~~~~~x\in \lbrack 0,1].
\end{align}
\parindent=8mmIn approximation theory, $q$-type generalization of Bernstein
polynomials was introduced by Lupa\c{s} \cite{n22}.

For $f \in C[0,1]$, the generalized Bernstein polynomial based on the $q$%
-integers is defined by Phillips \cite{n24} as follows
\begin{align}  \label{eq2}
{B}_{n,q}(f;x)=\sum\limits_{k=0}^{n}\left[
\begin{array}{c}
n \\
k%
\end{array}%
\right] _{q}x^{k}\prod\limits_{s=0}^{n-k-1}(1-q^{s}x)~~f\left( \frac{[k]_{q}
}{[n]_{q}}\right) ,~~x\in [0,1].
\end{align}

Recently, Mursaleen et al. \cite{n8} applied $(p,q)$-calculus in
approximation theory and introduced first $(p,q)$-analogue of Bernstein
operators(Revised) and defined as:

\begin{align}  \label{eq3}
{B}_{n,p,q}(f;x)=\frac{1}{p^{\frac{n(n-1)}{2}}}\sum_{k=0}^{n}f\left( \frac{[k]}{p^{k-n}[n]}\right)
P_{n,k}(p,q;x),~~~0<q<p\leq 1,~~x\in \lbrack 0,1]
\end{align}
where
\begin{equation}\label{e1q3}
P_{n,k}(p,q;x)=p^{\frac{k(k-1)}{2}}\left[
\begin{array}{c}
n \\
k%
\end{array}%
\right] _{p,q}x^{k}\prod_{s=0}^{n-k-1}(p^{s}-q^{s}x).
\end{equation}%
They have also introduced and studied approximation properties based on $%
(p,q)$-integers given as: Bernstein-Stancu operators \cite{n9} and
Bernstein-Shurer operators \cite{n19}.

We recall some basic properties of $(p,q)$-integers. \newline

The $(p,q)$-integer $[n]_{p,q}$ is defined by
\begin{equation*}
\lbrack n]_{p,q}=\frac{p^{n}-q^{n}}{p-q},~~~~~~~n=0,1,2,\cdots ,~~0<q<p\leq
1.
\end{equation*}%
The $(p,q)$-Binomial expansion is

\begin{equation*}
(x+y)_{p,q}^{n}:=(x+y)(px+qy)(p^{2}x+q^{2}y)\cdots (p^{n-1}x+q^{n-1}y)
\end{equation*}%
and the $(p,q)$-binomial coefficients are defined by

\begin{equation*}
\left[
\begin{array}{c}
n \\
k%
\end{array}%
\right] _{p,q}:=\frac{[n]_{p,q}!}{[k]_{p,q}![n-k]_{p,q}!}.
\end{equation*}%
For $p=1$, all the notions of $(p,q)$-calculus are reduced to $q$-calculus.
For details on $q$-calculus and $(p,q)$-calculus, one can refer \cite%
{n12,n13,n0,n5} and \cite{n22}, respectively. In this paper we use the notation $[n]$ in place of $[n]_{p,q}$.\newline

In \cite{n5}, $(p,q)$-derivative of a function $f(x)$ is defined by
\begin{align}  \label{eq4}
D_{p,q}f(x)=\frac{f(px)-f(qx)}{(p-q)x},~~~x\neq 0,
\end{align}
and the formulae for the $(p,q)$-derivative for the product of two
functions is given as

\begin{align}  \label{eq5}
D_{p,q}(fg)(x)=f(px).D_{p,q}g(x)+\{D_{p,q} f(x)\}.g(qx),
\end{align}
also
\begin{align}  \label{eq6}
D_{p,q}(fg)(x)=f(qx).D_{p,q}g(x)+\{D_{p,q} f(x)\}.g(px).
\end{align}

Let $r\in \mathbb{N}\cup \{0\}$ be a fixed number. For $f\in C^{r}[0,1]$ and
$m\in \mathbb{N}$, we define an operator of $r^{th}$ order for $(p,q)$-
Bernstein type operators as follows:
\begin{align}  \label{eq25}
B_{n,p,q}^{[r]}(f;x)=\frac{1}{p^{\frac{n(n-1)}{2}}}\sum_{k=0}^{n}P_{n,k}(p,q;x)\sum_{i=0}^{r}\frac{1}{i!}%
f^{(i)}\left( \frac{[k]}{p^{k-n}[n]}\right) \left( x-\frac{[k]}{p^{k-n}[n]}\right) ^{i}.
\end{align}

In this paper, using the moment estimates from \cite{n1}, we give the
estimates of the central moments for operators defined by \eqref{eq3}. We
also study some approximation properties of an $r^{th}$ order generalization
of the operators defined by \eqref{eq25} using the techniques of the work on
the higher order generalization of $q$-analogue \cite{n2}. Further, we study
approximation properties and prove Voronovskaja type theorem for these
operators.

\parindent=8mmIf we put $p=1$, then we get the moments for $q$-Bernstein
operators \cite{n1} and the usual generalization higher order $q$-Bernstein
operators \cite{n2}, respectively.

\section{Main results}
The following result is $(p,q)$-analogue of \cite{n0}.

\begin{proposition}\label{t1}
\label{t1} For $n\geq 1,~~0<q<p\leq 1$
\begin{align}  \label{eq7}
D_{p,q}(1+x)_{p,q}^n=[n](1+qx)_{p,q}^{n-1}.
\end{align}
\end{proposition}

\begin{proof}
By applying simple calculation on $(p,q)$-analogue, we have
\begin{align}  \label{eq8}
(1+px)_{p,q}^{n}=p^{n-1}(1+px)(1+qx)_{p,q}^{n-1},(1+qx)_{p,q}^{n}=(p^{n-1}+q^{n}x)(1+qx)_{p,q}^{n-1}.
\end{align}
Applying $(p,q)$-derivative and result \eqref{eq8} we get the desired result.
\end{proof}

\begin{lemma} \label{t2}
\label{t2}  Let $B_{n,p,q}(f;x)$ be given by \eqref{eq3}. Then for any $m
\in \mathbb{N},~~x \in [0,1] $ and $0<q<p\leq 1$ we have
\begin{eqnarray*}
B_{n,p,q}\left((t-x)_{p,q}^{m+1};x\right) &=&\frac{p^{m+n}x(1-x)}{[n]}
D_{p,q}\bigg\{B_{n,p,q}\left((t-\frac{x}{p})_{p,q}^{m};\frac{x}{p}\right)%
\bigg\} \\
&+&\frac{p^{m+n-1}[m]x(1-x)}{[n]} B_{n,p,q}\left((t-\frac{qx}{p}%
)_{p,q}^{m-1};\frac{qx}{p}\right) \\
&+&\frac{[m](p^n-q^n)x}{[n]} B_{n,p,q}\left((t-x)_{p,q}^{m};x\right).
\end{eqnarray*}
\end{lemma}

\begin{proof}
First of all by using \eqref{eq5} and Lemma \ref{t1}, we have\newline
$D_{p,q}\left(\frac{1}{p^{\frac{n(n-1)}{2}}} \sum_{k=0}^{n}\left( t-\frac{x}{p}\right)
_{p,q}^{m}P_{n,k}(p,q;\frac{x}{p})\right) $
\begin{align}  \label{eq100}
&=& \frac{1}{p^{\frac{n(n-1)}{2}}}\left(\sum_{k=0}^{n}\left( t-x\right)_{p,q}^{m}D_{p,q}\{P_{n,k}(p,q;\frac{x}{p%
})\}-\frac{[m]}{p}\sum_{k=0}^{n}\left( t-\frac{qx}{p}\right)
_{p,q}^{m-1}P_{n,k}(p,q;\frac{qx}{p})\right).
\end{align}%
Now in the same way by using \eqref{eq5} and Lemma \ref{t1}, we have \newline
$D_{p,q}\bigg\{P_{n,k}\left( p,q;\frac{x}{p}\right) \bigg\}
=D_{p,q}\bigg\{p^{\frac{k(k-1)}{2}} \lbrack k]
\left[
\begin{array}{c}
n \\
k%
\end{array}%
\right] _{p,q} 
\left(\frac{x}{p}\right)^k \left(1-\frac{x}{p}\right)^k \bigg\}$
\begin{align}  \label{eq125}
= p^{\frac{k(k-1)}{2}} \left(\lbrack k]
\left[
\begin{array}{c}
n \\
k%
\end{array}%
\right] _{p,q}\frac{1}{p^{k}}x^{k-1}\left( 1-\frac{qx}{p}\right)
_{p,q}^{n-k}-[n-k]\left[
\begin{array}{c}
n \\
k%
\end{array}%
\right] _{p,q}\frac{1}{p}x^{k}\left( 1-\frac{qx}{p}\right) _{p,q}^{n-k-1} \right).
\end{align}%

Now by a simple calculation, we have
\begin{align}  \label{eq126}
\left( 1-\frac{qx}{p}\right) _{p,q}^{n-k}=\frac{1}{p^{n-k}}%
(p-qx)_{p,q}^{n-k+1}=\frac{1}{p^{n-k}}\frac{1}{(1-x)}%
(p^{n-k}-q^{n-k}x)(1-x)_{p,q}^{n-k}
\end{align}%
\begin{align}  \label{eq127}
\left( 1-\frac{qx}{p}\right) _{p,q}^{n-k-1}=\frac{1}{p^{n-k-1}}\frac{1}{(1-x)%
}(1-x)_{p,q}^{n-k}.
\end{align}%
From \eqref{eq125},\eqref{eq126} and \eqref{eq127}, we get
\begin{equation*}
D_{p,q}\bigg\{P_{n,k}\left( p,q;\frac{x}{p}\right) \bigg\}=\frac{%
P_{n,k}(p,q;x)}{p^{n}x(1-x)}\left( [k](p^{n-k}-q^{n-k}x)-p^{k}[n-k]x\right) ,
\end{equation*}%
which implies that
\begin{align}  \label{eq128}
D_{p,q}\bigg\{P_{n,k}\left( p,q;\frac{x}{p}\right) \bigg\}=\frac{%
P_{n,k}(p,q;x)}{p^{n}x(1-x)}\left( p^{n-k}[k]-[n]x\right) .
\end{align}%
From \eqref{eq100}, \eqref{eq128}, we have

$D_{p,q} \left( \sum_{k=0}^n \left( t-\frac{x}{p}\right)_{p,q}^m P_{n,k}(p,q;%
\frac{x}{p})\right)$
\begin{eqnarray*}
&=& -\frac{1}{p^{\frac{n(n-1)}{2}}}\frac{[m]}{p}\sum_{k=0}^n \left( t-\frac{qx}{p}%
\right)_{p,q}^{m-1}P_{n,k}(p,q;\frac{qx}{p}) \\
&+&\frac{1}{p^{\frac{n(n-1)}{2}}}\frac{1}{p^{n}x(1-x)}\sum_{k=0}^n \left(t-x\right)_{p,q}^m
P_{n,k}(p,q;x)(p^{n-k}[k]-[n]x) \\
&=& -\frac{1}{p^{\frac{n(n-1)}{2}}}\frac{[m]}{p}\sum_{k=0}^n \left( t-\frac{qx}{p}%
\right)_{p,q}^{m-1}P_{n,k}(p,q;\frac{qx}{p}) \\
&+&\frac{1}{p^{\frac{n(n-1)}{2}}}\frac{1}{p^{n}x(1-x)}\sum_{k=0}^n \left(t-x\right)_{p,q}^m P_{n,k}(p,q;x)
\\
& \times & \left(\frac{[n]}{p^m} (p^{m}t-q^mx)-\frac{[n]}{p^m}(p^{m}-q^m)x\right).
\end{eqnarray*}%
\newline
Hence we have\newline
$D_{p,q} \left\{ B_{n,p,q} \left(( t-\frac{x}{p})_{p,q}^{m};\frac{x}{p}%
\right)\right\}$
\begin{eqnarray*}
&=& -\frac{[m]}{p} B_{n,p,q} \left(( t-\frac{qx}{p})_{p,q}^{m-1};\frac{qx}{p}%
\right) + \frac{[n]}{p^{m+n}x(1-x)}B_{n,p,q} \left( (
t-x)_{p,q}^{m+1};x\right) \\
&-&\frac{[m](p^n-q^n)}{p^{m+n}(1-x)}B_{n,p,q} \left( (
t-x)_{p,q}^{m};x\right).
\end{eqnarray*}
This complete the proof of Lemma \ref{t2}.
\end{proof}

\begin{lemma}\label{t3}
Let $B_{n,p,q}\left( (t-x)_{p,q}^{m};x\right) $ be a polynomial
in $x$ of degree less than or equal to $m$ and the minimum degree of $\frac{1%
}{[n]}$ is $\lfloor \frac{m+1}{2}\rfloor $. Then for any fixed $m\in \mathbb{%
N}$ and $x\in \lbrack 0,1],~~0<q<p\leq 1$ we have
\begin{align}\label{eq0}
B_{n,p,q}\left( (t-x)_{p,q}^{m};x\right) =\frac{x(1-x)}{[n]^{\lfloor \frac{%
m+1}{2}\rfloor }}\sum_{k=0}^{m-2}b_{k,m,n}(p,q)x^{k},
\end{align}%
such that the coefficients $b_{k,m,n}(p,q)$ satisfy $\mid b_{k,m,n}(p,q)\mid
\leq b_{m},~~k=1,2,\cdots ,m-2$ and $b_{m}$ does not depend on $x,t,p,q$;
where $\lfloor a\rfloor $ is an integer part of $a\geq 0.$
\end{lemma}


\begin{proof}
By induction it is true for $m=2$. Assuming it is true for $m$, then from
Lemma \ref{t2} and equation \eqref{eq0} we have\newline
\newline
$B_{n,p,q}\left( (t-x)_{p,q}^{m+1};x\right) $
\begin{eqnarray*}
&=&\frac{p^{m+n-1}x(1-x)}{{[n]}^{1+\lfloor \frac{m+1}{2}\rfloor }}%
D_{p,q}\left\{ x\left( 1-\frac{x}{p}\right)
\sum_{k=0}^{m-2}b_{k,m,n}(p,q)\left( \frac{x}{p}\right) ^{k}\right\}  \\
&+&\frac{p^{m+n-1}[m]{x(1-x)}}{{[n]}^{1+\lfloor \frac{m}{2}\rfloor }}%
\sum_{k=0}^{m-3}b_{k,m-1,n}(p,q)\left( \left( \frac{q}{p}\right)
^{k+1}x^{k+1}-\left( \frac{q}{p}\right) ^{k+2}x^{k+2}\right)  \\
&+&\frac{[m](p^{n}-q^{n}){x(1-x)}}{{[n]}^{1+\lfloor \frac{m+1}{2}\rfloor }}%
\sum_{k=0}^{m-2}b_{k,m,n}(p,q)x^{k+1} \\
&=&\frac{p^{m+n}{x(1-x)}}{{[n]}^{1+\lfloor \frac{m+1}{2}\rfloor }}%
\sum_{k=0}^{m-2}[k]b_{k,m,n}(p,q)\frac{1}{p^{k}}(x^{k}-x^{k+1}) \\
&+&\frac{p^{m+n-1}{x(1-x)}}{{[n]}^{1+\lfloor \frac{m+1}{2}\rfloor }}%
\sum_{k=0}^{m-2}b_{k,m,n}(p,q)\left( \frac{q}{p}\right) ^{k}\left( x^{k}-%
\frac{[2]}{p}x^{k+1}\right)  \\
&+&\frac{p^{m+n-1}[m]{x(1-x)}}{{[n]}^{1+\lfloor \frac{m}{2}\rfloor }}%
\sum_{k=0}^{m-3}b_{k,m-1,n}(p,q)\left( \left( \frac{q}{p}\right)
^{k+1}x^{k+1}-\left( \frac{q}{p}\right) ^{k+2}x^{k+2}\right)  \\
&+&\frac{[m](p^{n}-q^{n}){x(1-x)}}{{[n]}^{1+\lfloor \frac{m+1}{2}\rfloor }}%
\sum_{k=0}^{m-2}b_{k,m,n}(p,q)x^{k+1}
\end{eqnarray*}%
\begin{eqnarray*}
&=&\frac{{x(1-x)}}{{[n]}^{1+\lfloor \frac{m+1}{2}\rfloor }}%
\sum_{k=0}^{m-2}\left( p^{m+n-k}[k]+p^{m+n-k-1}q^{k}\right)
b_{k,m,n}(p,q)x^{k} \\
&-&\frac{{x(1-x)}}{{[n]}^{1+\lfloor \frac{m+1}{2}\rfloor }}%
\sum_{k=1}^{m-1}\left( p^{m+n+1-k}[k-1]_{p,q}+[2]p^{m+n-k-1}q^{k-1}\right)
b_{k-1,m,n}(p,q)x^{k} \\
&+&\frac{{x(1-x)}}{{[n]}^{1+\lfloor \frac{m}{2}\rfloor }}%
\sum_{k=1}^{m-2}[m]p^{m+n-k-1}q^{k}b_{k-1,m-1,n}(p,q)x^{k} \\
&-&\frac{{x(1-x)}}{{[n]}^{1+\lfloor \frac{m}{2}\rfloor }}%
\sum_{k=2}^{m-1}[m]p^{m+n-k-1}q^{k}b_{k-2,m-1,n}(p,q)x^{k} \\
&+&\frac{x(1-x)}{{[n]}^{1+\lfloor \frac{m+1}{2}\rfloor }}%
\sum_{k=1}^{m-1}[m](p^{n}-q^{n})b_{k-1,m,n}(p,q)x^{k}\\
&=&\frac{x(1-x)}{{[n]}^{\lfloor \frac{m+2}{2}\rfloor }}%
\sum_{k=0}^{m-1}b_{k,m+1,n}(p,q)x^{k}
\end{eqnarray*}%
where
\begin{eqnarray*}
b_{k,m+1,n}(p,q) &=&\frac{1}{[n]^{\alpha }}\left(
p^{m+n-k}[k]+p^{m+n-k-1}q^{k}\right) b_{k,m,n}(p,q) \\
&-&\frac{1}{[n]^{\alpha }}\left(
p^{m+n+1-k}[k-1]+[2]p^{m+n-k-1}q^{k-1}\right) b_{k-1,m,n}(p,q) \\
&+&\frac{1}{[n]^{\alpha }}%
[m](p^{n}-q^{n})b_{k-1,m,n}(p,q)+[m]p^{m+n-k-1}q^{k}b_{k-1,m-1,n}(p,q) \\
&-&[m]p^{m+n-k-1}q^{k}b_{k-2,m-1,n}(p,q). \\
\end{eqnarray*}%
Clearly
\begin{equation*}
\alpha =1+\lfloor \frac{m+1}{2}\rfloor -\lfloor \frac{m+2}{2}\rfloor
,~~~0\leq k\leq m-1,
\end{equation*}%
which lead us that either $\alpha =0$ or $\alpha =1$.\newline
Since $\mid b_{k,m,n}(p,q)\mid \leq b_{m}$, for $k=m-1$, clearly we have
\begin{eqnarray*}
\mid b_{k,m+1,n}(p,q)\mid  &\leq &\frac{1}{[n]^{\alpha }}\left(
p^{n+1}[m-1]+p^{n}q^{m-1}\right) b_{m}+\frac{1}{[n]^{\alpha }}\left(
p^{n+2}[m-2]+[2]p^{n}q^{m-2}\right) b_{m} \\
&+&\frac{1}{[n]^{\alpha }}[m](p^{n}-q^{n})b_{m}+[m]p^{n}q^{m-1}b_{m-1} \\
&+&[m]p^{n}q^{m-1}b_{m-1} \\
&=&\frac{1}{[n]^{\alpha }}\left( p[m-1]+q^{m-1}\right) b_{m}+\frac{1}{%
[n]^{\alpha }}\left( p^{2}[m-2]+[2]q^{m-2}\right) b_{m} \\
&+&\frac{1}{[n]^{\alpha }}[m]b_{m}+[m]q^{m-1}b_{m-1}+[m]q^{m-1}b_{m-1} \\
&=&b_{m+1},~~~~~k=1,2,\cdots m-1,
\end{eqnarray*}%
and $b_{m}$ does not depend on $x,t,p,q$. This complete the proof.
\end{proof}

From the Lemma \ref{t2} and Lemma \ref{t3} we have the following theorem.

\begin{theorem}\label{t4}
Let $m \in \mathbb{N}$ and $0<q<p\leq 1$. Then there exits a
constant $C_m >0$ such that for any $x\in [0,1]$, we have
\begin{equation*}
\mid B_{n,p,q}\left((t-x)_{p,q}^{m};x\right) \mid \leq C_m \frac{x(1-x)}{%
[n]^{\lfloor \frac{m+1}{2}\rfloor}}.
\end{equation*}
\end{theorem}

\begin{corollary}\label{t5}
Let $m \in \mathbb{N}$ and $0<q<p\leq 1$. Then there exits a
constant $K_m >0$ such that for any $x\in [0,1]$, we have
\begin{align}  \label{eq9}
B_{n,p,q}\left((\mid t-x\mid )_{p,q}^{m};x\right) \leq K_m \frac{x(1-x)}{%
[n]^{ \frac{m}{2}}}.
\end{align}
\end{corollary}

\begin{proof}
For an even $m,$ clearly we have
\begin{eqnarray*}
B_{n,p,q}\left( (\mid t-x\mid )_{p,q}^{m};x\right)  &=&B_{n,p,q}\left(
(t-x)_{p,q}^{m};x\right)  \\
&\leq &C_{m}\frac{x(1-x)}{[n]^{\lfloor \frac{m+1}{2}\rfloor }} \\
&=&K_{m}\frac{x(1-x)}{[n]^{\frac{m}{2}}}
\end{eqnarray*}%
In case if $m$ is odd, say $m=2u+1$, we have\newline
$B_{n,p,q}\left( (\mid t-x\mid )_{p,q}^{2u+1};x\right) $
\begin{eqnarray*}
&\leq &\sqrt{B_{n,p,q}\left( (\mid t-x\mid )_{p,q}^{4u};x\right) }\sqrt{%
B_{n,p,q}\left( (\mid t-x\mid )_{p,q}^{2};x\right) } \\
&\leq &\sqrt{C_{4u}\frac{x(1-x)}{[n]^{\lfloor \frac{4u+1}{2}\rfloor }}}\sqrt{%
C_{2}\frac{x(1-x)}{[n]^{\lfloor \frac{3}{2}\rfloor }}} \\
&=&\sqrt{C_{4u}\frac{x(1-x)}{[n]^{\frac{2u}{2}}}}\sqrt{C_{2}\frac{x(1-x)}{[n]%
}} \\
&=&K_{2u+1}\frac{x(1-x)}{[n]^{\frac{2u+1}{2}}}.
\end{eqnarray*}%
This complete the proof.
\end{proof}

\begin{theorem}\label{t7}
Let $B_{n,p,q}^{[r]}(f;x)$ be an operator from $%
C^{r}[0,1]\rightarrow C^{r}[0,1]$. Then for $0<q<p\leq 1$ there exits a
constant $M(r)$ such that for every $f\in C^{r}[0,1],$ we have
\begin{align}  \label{eq10}
\parallel B_{n,p,q}^{[r]}(f;x)\parallel _{C[0,1]}\leq
M(r)\sum_{i=0}^{r}\parallel f^{(i)}\parallel =M(r)\parallel f\parallel
_{C^{r}[0,1]}.
\end{align}
\end{theorem}

\begin{proof}
Clearly $B_{n,p,q}^{[r]}(f;x)$ is continuous on $[0,1]$. From \eqref{eq25}
we have
\begin{eqnarray*}
B_{n,p,q}^{[r]}(f;x) & = & \sum_{i=0}^{r}\frac{(-1)^i}{i!}
B_{n,p,q}\left((t-x)^if^{(i)}(t);x\right).
\end{eqnarray*}

From the Corollary \ref{t5}, we have

\begin{eqnarray*}
\mid B_{n,p,q}\left((t-x)^if^{(i)}(t);x\right) \mid & \leq & \parallel
f^{(i)}\parallel B_{n,p,q}\left(\mid (t-x) \mid ^i;x \right) \\
& \leq & K_i\parallel f^{(i)}\parallel [n]^{-\frac{i}{2}}.
\end{eqnarray*}
Therefore
\begin{eqnarray*}
\parallel B_{n,p,q}^{[r]}(f;x) \parallel & \leq & \sum_{i=0}^{r}\frac{(-1)^i%
}{i!}\parallel B_{n,p,q}\left((t-x)^if^{(i)}(t);x\right)\parallel \\
& \leq & M(r)\sum_{i=0}^{r}\parallel f^{(i)}\parallel.
\end{eqnarray*}
This complete the proof.
\end{proof}

\section{ Convergence properties of $B_{n,p,q}^{[r]}(f;x)$}

The modulus of continuity of the derivative $f^{(r)}$ is given by
\begin{align}  \label{eq11}
\omega \left( f^{(r)};t\right) =\sup \bigg\{\mid f^{(r)}(x)-f^{(r)}(y)\mid
:\mid x-y\mid \leq t,~~x,y\in \lbrack 0,1]\bigg\}.
\end{align}

\begin{theorem}\label{t7}
Let $0<q<p\leq 1$ and $r \in \mathbb{N}\cup \{0\}$ be a fixed
number. Then for $x \in [0,1],~~~n \in \mathbb{N}$ there exits $D_r>0$ such
that for every $f \in C^{r}[0,1]$ the following inequality holds
\begin{align}  \label{eq15}
\mid B_{n,p,q}^{[r]}(f;x)-f(x) \mid \leq D_r \frac{1}{[n]^{\frac{r}{2}}}
\omega \left( f^{(r)};\frac{1}{\sqrt{[n]}}\right).
\end{align}
\end{theorem}

\begin{proof}
Let $r\in \mathbb{N}$. Then for $f\in C^{r}[0,1]$ at a given point $t\in
\lbrack 0,1],$ we have from the Taylor formula that\newline
\begin{eqnarray*}
f(x) &=&\sum_{i=0}^{r}\frac{f^{(i)}(t)}{i!}(x-t)^{i}+\frac{(x-t)^{r}}{%
((r-1)!)} \\
&\times &\int_{0}^{1}(1-u)^{r-1}\left( f^{(r)}(t+u(x-t))-f^{(r)}(t)\right)
\mathrm{d}u.
\end{eqnarray*}%
On applying $B_{n,p,q}^{[r]}(f;x),$ we get
\begin{equation*}
f(x)-B_{n,p,q}^{[r]}(f;x)=\sum_{k=0}^{n}\frac{(x-\frac{[k]}{p^{k-n}[n]})^{r}}{(r-1)!%
}\int_{0}^{1}(1-u)^{r-1}P_{n,k}(p,q;x)
\end{equation*}%
\begin{align}  \label{eq12}
\times \left[ f^{(r)}\left( \frac{[k]}{p^{k-n}[n]}+u\left( x-\frac{[k]}{p^{k-n}[n]}\right)
\right) -f^{(r)}\left( \frac{[k]}{p^{k-n}[n]}\right) \right] \mathrm{d}u.
\end{align}

Now from the definition and properties of modulus of continuity, we have

\begin{equation*}
\bigg{|} f^{(r)}\left(\frac{[k]}{p^{k-n}[n]}+u\left(x-\frac{[k]}{p^{k-n}[n]}\right)\right)
-f^{(r)}\left(\frac{[k]}{p^{k-n}[n]}\right) \bigg{|}  \leq \omega \left( f^{(r)};u %
\bigg{|} x- \frac{[k]}{p^{k-n}[n]} \bigg{|}\right)
\end{equation*}

\begin{align}  \label{eq13}
\omega \left( f^{(r)};u \bigg{|} x- \frac{[k]}{p^{k-n}[n]} \bigg{|}\right) \leq
\left( \sqrt{[n]}\bigg{|} x- \frac{[k]}{p^{k-n}[n]} \bigg{|}+1\right) \omega \left(
f^{(r)};\frac{1}{\sqrt{[n]}}\right).
\end{align}
\end{proof}

Now for every $0\leq x\leq 1,~~0<q<p\leq 1,~~k\in \mathbb{N}\cup
\{0\},~~n\in \mathbb{N}$ and from \eqref{eq12} and \eqref{eq13}, we get%
\newline
$\mid B_{n,p,q}^{[r]}(f;x)-f(x)\mid $
\begin{equation*}
\leq \frac{1}{r!}\omega \left( f^{(r)};\frac{1}{\sqrt{[n]}}\right)
\sum_{k=0}^{n}\bigg{|}x-\frac{[k]}{p^{k-n}[n]}\bigg{|}^{r}\left( \sqrt{[n]}\bigg{|}%
x-\frac{[k]}{p^{k-n}[n]}\bigg{|}+1\right) P_{n,k}(p,q;x)
\end{equation*}%
\begin{align}  \label{eq14}
=\frac{1}{r!}\omega \left( f^{(r)};\frac{1}{\sqrt{[n]}}\right) \left( \sqrt{%
[n]}B_{n,p,q}\left( \mid x-t\mid ^{r+1};x\right) +B_{n,p,q}\left( \mid
x-t\mid ^{r};x\right) \right) .
\end{align}

Using \eqref{eq9} and \eqref{eq14} for $x\in \lbrack 0,1],$ we have

\begin{eqnarray*}
\mid B_{n,p,q}^{[r]}(f;x)-f(x) \mid &\leq & \frac{1}{r!}(K_{r+1}+K_r) \left(
\frac{1}{\sqrt{[n]}}\right)^r \omega \left( f^{(r)};\frac{1}{\sqrt{[n]}}%
\right) \\
& = & D_r \left( \frac{1}{\sqrt{[n]}}\right)^r \omega \left( f^{(r)};\frac{1%
}{\sqrt{[n]}}\right).
\end{eqnarray*}

\parindent=8mmIn order to obtain the uniform convergence of $%
B_{n,p_{n},q_{n}}^{[r]}(f;x)$ to a continuous function $f$, we take $%
q=q_{n},~~p=p_{n}$ where $q_{n}\in (0,1)$ and $p_{n}\in (q_{n},1]$
satisfying,
\begin{align}  \label{eq16}
\lim_{n}p_{n}=1,~~~~~~\lim_{n}q_{n}=1.
\end{align}

\begin{corollary}\label{t8}
Let $p=p_{n},~~q=q_{n},~~0<q_{n}<p_{n}\leq 1$ satisfy \eqref{eq16}
and $f\in C^{r}[0,1]$ for a fixed number $r\in \mathbb{N}\cup \{0\}$. Then
\begin{align}  \label{eq9}
\lim_{n\rightarrow \infty }[n]^{\frac{r}{2}}\parallel
B_{n,k}^{[r]}(f)-f\parallel =0.
\end{align}
\end{corollary}

We say that (cf. \cite{n2}) a function $f\in C[0,1]$ belongs to $%
Lip_{M}(\alpha ),$ $0<\alpha \leq 1$, provided
\begin{align}  \label{eq18}
\mid f(x)-f(y)\mid \leq M\mid x-y\mid ^{\alpha },~~~(x,y\in \lbrack 0,1]~~%
\mbox{and}~~M>0).
\end{align}

\begin{corollary}\label{t9}
Let $p=p_{n},~~q=q_{n},~~0<q_{n}<p_{n}\leq 1$ satisfy \eqref{eq16}
and $f\in C^{r}[0,1]$ for a fixed number $r\in \mathbb{N}\cup \{0\}$. If $%
f^{(r)}\in \mbox{Lip}_{M}(\alpha )$ then
\begin{align}  \label{eq9}
\parallel B_{n,p,q}^{[r]}(f)-f\parallel =O\left( [n]^{-\frac{r+\alpha }{2}%
}\right) .
\end{align}
\end{corollary}

\begin{proof}
From \eqref{eq15} and \eqref{eq18}, we have
\begin{equation*}
\parallel B_{n,p,q}^{[r]}(f)-f\parallel \leq D_{r}M\frac{1}{[n]^{\frac{r}{2}}%
}\frac{1}{[n]^{\frac{\alpha }{2}}}.
\end{equation*}
\end{proof}

\begin{theorem}\label{t10}
Let $0<q<p \leq 1$. Suppose that $f \in C^{r+2}[0,1]$, where $r
\in \mathbb{N}\cup \{0\}$ is fixed then we have\newline
\begin{equation*}
\bigg{|} B_{n,p,q}^{[r]}(f;x)-f(x)-\frac{(-1)^rf^{(r+1)}(x)
B_{n,p,q}\left((t-x)^{r+1};x\right)}{(r+1)!}
\end{equation*}
\begin{eqnarray*}
&-& \frac{(-1)^rf^{(r+2)}(x) B_{n,p,q}\left((t-x)^{r+2};x\right)}{(r+2)!}%
\bigg{|} \\
&\leq & (K_{r+2}+K_{r+4}) \frac{x(1-x)}{[n]^{\frac{r}{2}+1}}\sum_{i=0}^r
\frac{1}{i!(r+2-i)!}\omega \left( f^{(r+2-i)}, [n]^{-\frac{1}{2}} \right).
\end{eqnarray*}
\end{theorem}

\begin{proof}
Let $f \in C^{r+2}[0,1]$ and $x \in [0,1]$ for a fixed number $r \in \mathbb{%
N}\cup\{0\}$ we have $f^{(i)} \in C^{r+2-i}[0,1],~~0\leq i \leq r$. Then by
Taylor formula we can write
\begin{align}  \label{eq20}
f^{(i)}(t)=\sum_{i=0}^{r+2-i}\frac{f^{(i+j)}(x)}{j!}(t-x)^j+R_{r+2-j}(f;t;x),
\end{align}
where
\begin{equation*}
R_{r+2-i}(f;t;x)=\frac{{f^{(r+2-i)}(\zeta_{p^{n-k-1}t})}-f^{(r+2-i)}(x)}{%
(r+2-i)!}(t-x)^{r+2-i},
\end{equation*}
and
\begin{equation*}
\mid \zeta_{t}-x \mid < \mid t-x \mid.
\end{equation*}
Therefore from \eqref{eq25} and \eqref{eq20} we have

\begin{eqnarray*}
B_{n,p,q}^{[r]}(f;x) &= & \sum_{k=0}^{n}P_{n,k}(p,q;x)\sum_{i=0}^r \frac{%
\left(x-\frac{[k]}{p^{k-n}[n]}\right)^i}{i!} \sum_{j=0}^{r+2-i} \frac{f^{(i+j)}(x)}{%
j!} \left(\frac{[k]}{p^{k-n}[n]}-x\right)^j \\
& + & \sum_{k=0}^{n}P_{n,k}(p,q;x)\sum_{i=0}^r \frac{\left(x-\frac{[k]}{p^{k-n}[n]}%
\right)^i}{i!} R_{r+2-i}(f;t;x) \\
& = & I_1+I_2,~~~~\{\mbox{where}~~~t=\frac{[k]}{p^{k-n}[n]}\}
\end{eqnarray*}
Which implies that\newline
\newline
$\mid B_{n,p,q}^{[r]}(f;x) - I_1 \mid$
\begin{eqnarray*}
&=&\mid I_2 \mid \\
&= &\bigg{|} \sum_{k=0}^{n}P_{n,k}(p,q;x)\sum_{i=0}^r \frac{(-1)^i}{i!}
\frac{{f^{(r+2-i)}(\zeta_{t} )}-f^{(r+2-i)}(x)}{(r+2-i)!}(t-x)^{r+2}\bigg{|}
\\
&= & \bigg{|} B_{n,p,q} \left(\sum_{i=0}^r \frac{(-1)^i}{i!} \frac{{%
f^{(r+2-i)}(\zeta_{t} )}-f^{(r+2-i)}(x)}{(r+2-i)!}(t-x)^{r+2} ;x\right)%
\bigg{|}.
\end{eqnarray*}
We use the well-known inequality
\begin{equation*}
\omega (f,\lambda \delta) \leq (1+\lambda^2)\omega (f, \delta),
\end{equation*}
\begin{eqnarray*}
\mid {f^{(r+2-i)}(\zeta_{t} )}-f^{(r+2-i)}(x)\mid & \leq & \omega \left(
f^{(r+2-i)}, \mid \zeta_{t}-x\mid \right) \\
& \leq & \omega \left( f^{(r+2-i)}, \mid t-x\mid \right) \\
& \leq & \omega \left( f^{(r+2-i)}, [n]^{-\frac{1}{2}} \right) \left( 1+ [n]
(t-x)^2\right). \\
\end{eqnarray*}
Hence \newline
$\mid I_2\mid 
\leq \bigg{|} B_{n,p,q} \left(\sum_{i=0}^r \frac{(-1)^i}{i!} \frac{{%
f^{(r+2-i)}(\zeta_{t} )}-f^{(r+2-i)}(x)}{(r+2-i)!} \bigg{|}~~ \mid t-x
\mid^{r+2} ;x\right) $\newline
$\leq B_{n,p,q} \left(\sum_{i=0}^r \frac{1}{i!(r+2-i)!}\omega \left(
f^{(r+2-i)}, [n]^{-\frac{1}{2}} \right) \left( 1+ [n] (t-x)^2\right)  \mid
t-x \mid^{r+2} ;x\right)$\newline
$= \sum_{i=0}^r \frac{1}{i!(r+2-i)!}\omega \left( f^{(r+2-i)}, [n]^{-\frac{1%
}{2}} \right)$\newline
$\times \left( B_{n,p,q}( \mid t-x \mid^{r+2};x) + [n] B_{n,p,q}( \mid t-x
\mid^{r+4};x) \right)$\newline
$\leq \sum_{i=0}^r \frac{1}{i!(r+2-i)!}\omega \left( f^{(r+2-i)}, [n]^{-%
\frac{1}{2}} \right)  \left(K_{r+2} \frac{x(1-x)}{[n]^{\frac{r}{2}+1}}+
K_{r+4} \frac{x(1-x)}{[n]^{\frac{r}{2}+1}}\right)$\newline
$= (K_{r+2}+K_{r+4}) \frac{x(1-x)}{[n]^{\frac{r}{2}+1}}\sum_{i=0}^r \frac{1}{%
i!(r+2-i)!}\omega \left( f^{(r+2-i)}, [n]^{-\frac{1}{2}} \right)$. \newline
Therefore
\begin{equation*}
\mid B_{n,p,q}^{[r]}(f;x) - I_1 \mid \leq (K_{r+2}+K_{r+4}) \frac{x(1-x)}{%
[n]^{\frac{r}{2}+1}}\sum_{i=0}^r \frac{1}{i!(r+2-i)!}\omega \left(
f^{(r+2-i)}, [n]^{-\frac{1}{2}} \right).
\end{equation*}
Now we simplify for $I_1$
\begin{eqnarray*}
I_1 & = & \sum_{k=0}^{n}P_{n,k}(p,q;x)\sum_{i=0}^r \frac{(x-\frac{[k]}{p^{k-n}[n]}%
)^i}{i!} \sum_{l=i}^{r+2} \frac{f^{(l)}(x)}{(l-i)!} \left(\frac{[k]}{p^{k-n}[n]}%
-x\right)^{l-i} \\
& = & \sum_{k=0}^{n}P_{n,k}(p,q;x)\sum_{i=0}^r \frac{(-1)^i}{i!}
\sum_{l=i}^{r} \frac{f^{(l)}(x)}{(l-i)!} \left(\frac{[k]}{p^{k-n}[n]}-x\right)^{l}
\\
& + & \sum_{k=0}^{n}P_{n,k}(p,q;x)\sum_{i=0}^r \frac{(-1)^i}{i!} \frac{%
f^{(r+1)}(x)}{(r+1-i)!}\left(\frac{[k]}{p^{k-n}[n]}-x\right)^{r+1} \\
& + &\sum_{k=0}^{n}P_{n,k}(p,q;x)\sum_{i=0}^r \frac{(-1)^i}{i!}\frac{%
f^{(r+2)}(x)}{(r+2-i)!}\left(\frac{[k]}{p^{k-n}[n]}-x\right)^{r+2} \\
& = & \sum_{k=0}^{n}P_{n,k}(p,q;x) \sum_{l=0}^{r} \frac{f^{(l)}(x)}{(l)!}
\left(\frac{[k]}{p^{k-n}[n]}-x\right)^{l} \sum_{i=0}^{l}\binom{l}{i}(-1)^i \\
& + & \frac{f^{(r+1)}(x)}{(r+1)!}\sum_{k=0}^{n}P_{n,k}(p,q;x) \left(\frac{[k]%
}{p^{k-n}[n]}-x\right)^{r+1} \sum_{i=0}^{r}\binom{r+1}{i}(-1)^i \\
& + & \frac{f^{(r+2)}(x)}{(r+2)!}\sum_{k=0}^{n}P_{n,k}(p,q;x) \left(\frac{[k]%
}{p^{k-n}[n]}-x\right)^{r+2} \sum_{i=0}^{r}\binom{r+2}{i}(-1)^i. \\
\end{eqnarray*}
For $n \in \mathbb{N},~~~r \in \mathbb{N}\cup \{0\}$ we have
\begin{equation*}
\sum_{i=0}^{r}\binom{r+1}{i}(-1)^i=(-1)^r,~~~\sum_{i=0}^{r}\binom{r+2}{i}%
(-1)^i=(r+1)(-1)^r
\end{equation*}
Therefore
\begin{eqnarray*}
I_1 & = & f(x)+\frac{(-1)^rf^{(r+1)}(x) B_{n,p,q}\left((t-x)^{r+1};x\right)}{%
(r+1)!} \\
& + & \frac{(-1)^rf^{(r+2)}(x) B_{n,p,q}\left((t-x)^{r+2};x\right)}{(r+2)!}.
\end{eqnarray*}
This complete the proof.
\end{proof}

\begin{corollary}
Let $p=p_{n},~~q=q_{n},~~0<q_{n}<p_{n}\leq 1$ satisfy \eqref{eq16} and $f\in
C^{2}[0,1]$ for a fixed number $r\in \mathbb{N}\cup \{0\}$. Then for every $%
x\in \lbrack 0,1]$ we have
\begin{equation*}
\bigg{|}B_{n,p_{n},q_{n}}^{[r]}(f;x)-f(x)-\frac{f^{\prime \prime }(x)}{2}%
\frac{x(1-x)}{[n]}\bigg{|}\leq K\frac{x(1-x)}{[n]}\omega \left( f^{\prime
\prime}, [n]^{-\frac{1}{2}}\right) ,
\end{equation*}%
where $K=\frac{K_{2}+K_{4}}{2}$. Moreover,
\begin{equation*}
\lim_{n\rightarrow \infty }[n]\left( B_{n,p_{n},q_{n}}(f;x)-f(x)\right) =%
\frac{x(1-x)}{2}f^{\prime \prime }(x)
\end{equation*}%
uniformly on $[0,1]$.
\end{corollary}

\textbf{Acknowledgement.} Second author (MN) acknowledges the financial
support of University Grants Commission (Govt. of Ind.) for awarding BSR
(Basic Scientific Research) Fellowship.


\begin{thebibliography}{99}
\bibitem{n0} A. Aral, V. Gupta and R. P. Agarwal, Application of $q$%
-Calculus in Operator Theory, Springer, New York, 2013. .

\bibitem{n21} S.N. Bernstein, D\'{e}mostration du th\'{e}or\`{e}eme de
Weierstrass fond\'{e}e sur le calcul de probabilit\'{e}s, Comm. Soc. Math.
Kharkow (2), 13 (1912/1913) 1-2.



\bibitem{n5} M.N. Hounkonnou, J. D\'{e}sir\'{e} and B. Kyemba, $\mathcal{R}%
(p,q)$-calculus: differentiation and integration, SUT Jour. Math., 49(2)
(2013) 145-167.

\bibitem{n6} B. Lenze, Bernstein-Baskakov-Kantorovich operators and
Lipschitz-type maximal functions, in: Colloq. Math. Soc. Janos Bolyai, 58,
Approx. Th., (1990) 469-496.


\bibitem{n22} A. Lupa\c{s}, A $q$-analogue of the Bernstein operator,
Seminar on Numerical and Statistical Calculus, University of Cluj-Napoca,
9(1987) 85-92.

\bibitem{n1} N. Mahmudov, The moments for $q$-Bernstein operators in the
case $0<q<1$. Numer Algor (2010) 53:439--450, DOI 10.1007/s11075-009-9312-1.

\bibitem{n8} M. Mursaleen, K.J. Ansari and A. Khan, On $(p,q)$-analogue of
Bernstein operators, Appl. Math. Comput., 266(2015), 874-882.

\bibitem{n9} M. Mursaleen, K.J. Ansari and A. Khan, Some approximation
results by $(p,q)$-analogue of Bernstein-Stancu operators, Appl. Math.
Comput., 264 (2015) 392-402.


\bibitem{n19} M. Mursaleen, Md. Nasiuzzaman and Ashirbayev Nurgali, Some
approximation results on Bernstein-Schurer operators defined by $(p,q)$%
-integers, Jou. Ineq. Appl.\textbf{, }2015 (2015):249.\textit{\ }

%


\bibitem{n24} G. M. Phillips, Bernstein polynomials based on the $q$%
-integers, Ann. Numer. Math., 4 (1997), 511--518.

\bibitem{n12} P.N. Sadjang, On the fundamental theorem of $(p,q)$-calculus
and some $(p,q)$-Taylor formulas, arXiv:1309.3934 [math.QA].

\bibitem{n13} V. Sahai and S. Yadav, Representations of two parameter
quantum algebras and $(p,q)$-special functions, J. Math. Anal. Appl. 335
(2007) 268-279.


\bibitem{n2} P. Sabanc{\i }gil, Higher order generalization of $q$-Bernstein
operators, Jour. Comput. Analy. Appl., 12 (2010) 821-827.
\end{thebibliography}
\end{document}